\documentclass[a4paper]{amsart}
\usepackage{amsmath, amsthm, amsfonts, mathtools, tikz-cd, amssymb}
\usepackage{enumitem}
\usepackage{hyperref}

\title[A note on higher coherence of graphs of groups]{A note on higher coherence of graphs of groups}
\date{September 2, 2026.}


\author[K.~Li]{Kevin Li}
\address{Institut f\"ur Mathematik, Freie Universit\"at Berlin, 14195 Berlin, Germany}
\email{kevin.li@fu-berlin.de}

\author[L.J.~S\'anchez Salda\~na]{Luis Jorge S\'anchez Salda\~na}
\address{Departamento de Matem\'aticas, Facultad de Ciencias, Universidad Nacional Aut\'onoma de M\'exico, 04510 Ciudad de M\'exico, Mexico}
\email{luisjorge@ciencias.unam.mx}

\theoremstyle{definition}
\newtheorem{defn}{Definition}[section]
\newtheorem{ex}[defn]{Example}
\newtheorem{rem}[defn]{Remark}
\newtheorem*{ack}{Acknowledgements}
\theoremstyle{plain}
\newtheorem{thm}[defn]{Theorem}
\newtheorem{lem}[defn]{Lemma}

\newtheorem{cor}[defn]{Corollary}

\newcommand{\enum}{\rm{(\roman*)}}

\newcommand{\IN}{\ensuremath{\mathbb{N}}}
\newcommand{\IZ}{\ensuremath{\mathbb{Z}}}
\newcommand{\IQ}{\ensuremath{\mathbb{Q}}}
\newcommand{\sfF}{\ensuremath{\mathsf{F}}}
\newcommand{\sfFP}{\ensuremath{\mathsf{FP}}}
\newcommand{\calT}{\ensuremath{\mathcal{T}}}

\DeclareMathOperator{\cd}{cd}

\begin{document}

\begin{abstract}
	For~$n\in \mathbb{N}$, a group is called $n$-coherent if every subgroup of type~$\mathsf{F}_n$ is of type~$\mathsf{F}_{n+1}$.
	For~$n\ge 1$, we observe that graphs of groups with $n$-coherent vertex groups and virtually poly-cyclic edge groups are $n$-coherent.
	We deduce the $n$-coherence of certain right-angled Artin groups.
\end{abstract}

\maketitle

\section{Introduction}

A group is said to be \emph{coherent} if every finitely generated subgroup is finitely presented.
This notion has many connections, e.g., to low-dimensional topology, geometric group theory, graph theory, and ring theory~\cite{Wise_survey}.
The subject of this note is a higher version of coherence.

For~$n\in \IN$, a group~$G$ is \emph{of type~$\sfF_n$} (resp.\ \emph{of type~$\sfF_\infty$}) if there exists a CW-model for the classifying space~$BG$ with finite $n$-skeleton (resp.\ of finite type).
Every group is of type~$\sfF_0$.
For a group being of type~$\sfF_1$ (resp.\ of type~$\sfF_2$) is equivalent to being finitely generated (resp.\ finitely presented).

\begin{defn}
\label{defn:higher coherence}
	Let~$n,m\in \IN\cup \{\infty\}$ with $n<m$.
	We say that a group~$G$ is \emph{$(n,m)$-coherent} if every subgroup of~$G$ that is of type~$\sfF_n$ is of type~$\sfF_m$.
\end{defn}

Definition~\ref{defn:higher coherence} generalises the usual notions: $(0,1)$-coherence is being Noetherian (also called slender), $(1,2)$-coherence is coherence, and $(n,n+1)$-coherence is recently called $n$-coherence~\cite{Kochloukova-Vidussi23, Vidussi23, Kochloukova-Vidussi25}.
Clearly, $(n,m)$-coherence implies $(n',m')$-coherence for $n\le n'< m'\le m$ and is preserved under taking subgroups.
Moreover, $(n,m)$-coherence passes to finite index overgroups and hence is invariant under commensurability.
Virtually poly-cyclic groups are $(0,\infty)$-coherent.
Groups of geometric dimension~$\le n$ are $(n,\infty)$-coherent.
The direct product of free groups~$(F_2)^k$ is $(n,n+1)$-coherent if and only if~$n\ge k$.

The following is a classical combination theorem for coherence of graphs of groups.

\begin{thm}[{Karrass--Solitar~\cite{KS70}}]
\label{thm:KS}
	Let $G$ be the fundamental group of a graph of groups whose vertex groups are $(1,2)$-coherent and whose edge groups are $(0,1)$-coherent.
	Then~$G$ is $(1,2)$-coherent.
\end{thm}

We prove a combination theorem for higher coherence of graphs of groups.
We use the convention that $\infty-1=\infty$.

\begin{thm}
\label{thm:gog_coherence}
	Let~$n,m\in \IN\cup \{\infty\}$ with $1\le n< m$.
	Let~$G$ be the fundamental group of a graph of groups whose vertex groups are $(n,m)$-coherent and whose edge groups are
	\[
		\begin{cases}
			(0,m-1)\text{-coherent} & \text{if } n=1;
			\\
			(1,m-1)\text{-coherent} & \text{if }n\ge 2.
		\end{cases}
	\]
	Then~$G$ is $(n,m)$-coherent.
\end{thm}

Theorem~\ref{thm:KS} was crucial to characterise the coherent right-angled Artin groups.
 A simplicial complex is \emph{flag} if every complete subgraph in its 1-skeleton spans a simplex.
Associated to a finite flag simplicial complex~$L$ is the \emph{right-angled Artin group~$A_L$} that has as generating set the vertices of~$L$, subject only to the relations that two generators commute if the corresponding vertices are adjacent in~$L$~\cite{Koberda_survey}.
A graph is \emph{chordal} if it does not contain a cycle of length~$\ge 4$ as a full subgraph.
It is a classical result of Dirac~\cite{Dirac61} that in chordal graphs, minimal separating sets are cliques. 
It follows that a flag simplicial complex with chordal $1$-skeleton is a ``tree of simplices" in the following sense.

\begin{defn}
	We define~$\calT_1$ to be the smallest class of finite flag simplicial complexes that satisfies the following:
	\begin{enumerate}[label=\enum]
		\item All simplices (of all dimensions) lie in~$\calT_1$;
		\item If~$L_1,L_2\in \calT_1$ and~$L_0$ is a full subcomplex of both~$L_1$ and~$L_2$ that is a simplex or is empty, then the gluing~$L_1\cup_{L_0} L_2$ lies in~$\calT_1$.
	\end{enumerate}
\end{defn}

\begin{thm}[Droms~\cite{Droms87}]
\label{thm:Droms}
	Let~$L$ be a finite flag simplicial complex.
	Then the following are equivalent:
	\begin{enumerate}[label=\enum]
		\item The group~$A_L$ is $(1,\infty)$-coherent;		
		\item The group~$A_L$ is $(1,2)$-coherent;
		\item The $1$-skeleton of~$L$ is a chordal graph;
		\item The simplicial complex~$L$ lies in the class~$\calT_1$.
	\end{enumerate}
\end{thm}

As an application of Theorem~\ref{thm:gog_coherence}, we deduce the higher coherence of certain right-angled Artin groups.
We consider similar classes of simplicial complexes that allow for more vertex spaces and for edge spaces in~$\calT_1$.

\begin{defn}
\label{defn:class}
	Let~$n\in \IN$ with~$n\ge 2$.
	We define~$\calT_n$ to be the smallest class of finite flag simplicial complexes that satisfies the following:
	\begin{enumerate}[label=\enum]
		\item\label{item:simplices}
		All simplices (of all dimensions) lie in~$\calT_n$;
		\item\label{item:dimension}
		All finite flag simplicial complexes of dimension~$\le n-1$ lie in~$\calT_n$;
		\item\label{item:Betti}
		All finite flag simplicial complexes~$L$ of dimension~$n$ with $H_n(L;\IQ)=0$ lie in~$\calT_n$;
		\item\label{item:limit}
		All finite flag simplicial complexes that are joins~$L_1\ast\cdots\ast L_k$, where~$k\le n$ and~$L_1,\ldots,L_k$ are disjoint unions of simplices, lie in~$\calT_n$;
		\item\label{item:gluing}
		If~$L_1,L_2\in \calT_n$ and~$L_0$ is a full subcomplex of both~$L_1$ and~$L_2$ that lies in~$\calT_1$ or is empty, then the gluing~$L_1\cup_{L_0} L_2$ lies in~$\calT_n$;
		\item\label{item:cone}
		If~$L\in \calT_n$, then the cone~$L\ast \Delta^0$ lies in~$\calT_n$.
	\end{enumerate}
\end{defn}

\begin{thm}
\label{thm:RAAG}
	Let~$n\in \IN$ with~$n\ge 2$ and let~$L$ be a finite flag simplicial complex.
	If~$L\in \calT_n$, then the right-angled Artin group~$A_L$ is $(n,\infty)$-coherent.
\end{thm}

We do not expect the converse of Theorem~\ref{thm:RAAG} to hold.
This leaves open a characterisation of higher coherence for right-angled Artin groups.

\section{Inheritance of higher coherence}

We recall some well-known inheritance results for finiteness properties~\cite{Bieri81,Brown82,Geoghegan08,Haglund-Wise21}.

\begin{thm}
\label{thm:gog}
	Let~$G$ be the fundamental group of a finite graph of groups and let~$n\in \IN\cup \{\infty\}$.
	The following hold:
	\begin{enumerate}[label=\enum]
		\item\label{item:gog}
		If all vertex groups are of type~$\sfF_n$ and all edge groups are of type~$\sfF_{n-1}$, then~$G$ is of type~$\sfF_n$;
		\item\label{item:vertex}
		If~$G$ is of type~$\sfF_n$ and all edge groups are of type~$\sfF_n$, then all vertex groups are of type~$\sfF_n$.
	\end{enumerate}
\end{thm}

By Bass--Serre theory, there is a correspondence between graphs of groups and groups acting on trees (without inversions), under which vertex and edge groups correspond to vertex and edge stabilisers, respectively, of orbit representatives.

\begin{thm}[{\cite[Corollary~2.19]{GL}}]
\label{thm:GL}
	Let~$G$ be a group and let~$T$ be a cocompact $G$-tree.
	If~$G$ is finitely presented, then there exists a cocompact $G$-tree~$S$ with finitely generated edge stabilisers together with a $G$-map~$S\to T$ that sends vertices to vertices and edges to non-trivial edge paths.
	In particular, the vertex and edge stabilisers of~$S$ are subgroups of vertex and edge stabilisers of~$T$, respectively.
\end{thm}

\begin{thm}
\label{thm:extension}
	Let~$1\to N\to G\to Q\to 1$ be a group extension and let~$n\in \IN\cup \{\infty\}$.
	The following hold:
	\begin{enumerate}[label=\enum]
		\item\label{item:extension}
		If~$Q$ is of type~$\sfF_n$ and~$N$ is of type~$\sfF_n$, then~$G$ is of type~$\sfF_n$;
		\item\label{item:quotient}
		If~$G$ is of type~$\sfF_n$ and~$N$ is of type~$\sfF_{n-1}$, then~$Q$ is of type~$\sfF_n$.
	\end{enumerate}
\end{thm}

These inheritance results for finiteness properties give rise to corresponding inheritance results for higher coherence.

\begin{proof}[Proof of Theorem~\ref{thm:gog_coherence}]
	Let~$H$ be a subgroup of~$G$ that is of type~$\sfF_n$.
	We show that~$H$ is of type~$\sfF_{m}$.
	Consider the Bass--Serre tree~$T$ associated to~$G$ as an $H$-tree by restriction.
	Since the stabilisers of the $H$-tree~$T$ are subgroups of the stabilisers of the $G$-tree~$T$, they inherit the corresponding coherence properties.
	Since~$H$ is finitely generated, there exists a cocompact $H$-subtree~$T'$ of~$T$.
	In particular, the vertex and edge stabilisers of the $H$-tree~$T'$ are subgroups of vertex and edge stabilisers of the $H$-tree~$T$, respectively.
	
	We treat the two cases $n=1$ and $n\ge 2$ separately.
	First, let~$n=1$.
	Since the edge stabilisers of~$T$ are $(0,m-1)$-coherent, the edge stabilisers of~$T'$ are of type~$\sfF_{m-1}$ and, in particular, of type~$\sfF_1$.
	By Theorem~\ref{thm:gog}~\ref{item:vertex}, the vertex stabilisers of~$T'$ are of type~$\sfF_1$.
	Since the vertex stabilisers of~$T$ are $(1,m)$-coherent, the vertex stabilisers of~$T'$ are of type~$\sfF_m$.
	By Theorem~\ref{thm:gog}~\ref{item:gog}, it follows that~$H$ is of type~$\sfF_m$.
	
	Second, let~$n\ge 2$.
	By Theorem~\ref{thm:GL}, there exists a cocompact $H$-tree~$S$ whose edge stabilisers are of type~$\sfF_1$ and whose vertex and edge stabilisers are subgroups of vertex and edge stabilisers of~$T'$, respectively.
	 Since the edge stabilisers of~$T$ are $(1,m-1)$-coherent, the finitely generated edge stabilisers of~$S$ are of type~$\sfF_{m-1}$ and, in particular, of type~$\sfF_n$.
	 By Theorem~\ref{thm:gog}~\ref{item:vertex}, the vertex stabilisers of~$S$ are of type~$\sfF_n$.
	Since the vertex stabilisers of~$T$ are $(n,m)$-coherent, the vertex stabilisers of~$S$ are of type~$\sfF_m$.
	By Theorem~\ref{thm:gog}~\ref{item:gog}, it follows that~$H$ is of type~$\sfF_m$.
\end{proof}

\begin{lem}
\label{lem:extension_coherence}
	Let~$1\to N\to G\to Q\to 1$ be a group extension and let~$n,m\in \IN\cup \{\infty\}$ with~$n<m$.
	The following hold:
	\begin{enumerate}[label=\enum]
		\item\label{item:extension_coherence}
		If~$Q$ is $(n,m)$-coherent and~$N$ is $(0,m)$-coherent, then~$G$ is $(n,m)$-coherent;
		\item If~$G$ is $(n,m)$-coherent and~$N$ is of type~$\sfF_{m-1}$, then~$Q$ is $(n,m)$-coherent.
	\end{enumerate}
	\begin{proof}
		Let~$p\colon G\to Q$ denote the quotient map.
		For part~(i), let~$H$ be a subgroup of~$G$ that is of type~$\sfF_n$.
		Consider the group extension 
		\[
			1\to N\cap H\to H\to p(H)\to 1.
		\]
		Since~$N$ is $(0,m)$-coherent, the group~$N\cap H$ is of type~$\sfF_m$ and, in particular, of type~$\sfF_{n-1}$.
		By Theorem~\ref{thm:extension}~\ref{item:quotient}, the group~$p(H)$ is of type~$\sfF_n$.
		Since~$Q$ is $(n,m)$-coherent, the group~$p(H)$ is of type~$\sfF_m$.
		By Theorem~\ref{thm:extension}~\ref{item:extension}, it follows that~$H$ is of type~$\sfF_m$.

		For part~(ii), let~$K$ be a subgroup of~$Q$ that is of type~$\sfF_n$.
		Similarly, consider the group extension $1\to N\to p^{-1}(K)\to K\to 1$ and apply Theorem~\ref{thm:extension}~\ref{item:extension} and~\ref{item:quotient}.
	\end{proof}
\end{lem}

\begin{rem}
	In Theorem~\ref{thm:gog_coherence} for~$n=1$, instead of assuming that the edge groups are $(0,m-1)$-coherent, it suffices that the intersection of any subgroup of~$G$ that is of type~$\sfF_n$ with any edge group is of type~$\sfF_{m-1}$.
	Similarly, in Lemma~\ref{lem:extension_coherence}~\ref{item:extension_coherence}, instead of assuming that~$N$ is $(0,m)$-coherent, it suffices that the intersection of any subgroup of~$G$ that is of type~$\sfF_n$ with~$N$ is of type~$\sfF_m$.
\end{rem}

\section{Higher homological coherence}

There is an analogous notion of higher homological coherence~\cite{Kochloukova-Vidussi23, Fisher24}.
	Let~$R$ be a ring and let~$n\in \IN$.
	A group~$G$ is \emph{of type~$\sfFP_n(R)$} (resp.\ of type~$\sfFP_\infty(R)$) if the trivial $RG$-module~$R$ admits a projective resolution that is finitely generated in degrees~$\le n$ (resp.\ in all degrees).
	
\begin{defn}	
	Let~$R$ be a ring and let~$n,m\in \IN\cup \{\infty\}$ with~$n<m$.
	We say that a group~$G$ is \emph{homologically $(n,m)$-coherent over~$R$} if every subgroup of~$G$ that is of type~$\sfFP_n(R)$ is of type~$\sfFP_m(R)$.
\end{defn}	

	For~$n\ge 2$, homological $(n,m)$-coherence over~$\IZ$ implies $(n,m)$-coherence because for~$n\ge 2$, a group is of type~$\sfF_n$ if and only if it is of type~$\sfF_2$ and of type~$\sfFP_n(\IZ)$.
	Inheritance results similar to Theorem~\ref{thm:gog_coherence} and Lemma~\ref{lem:extension_coherence} hold also for higher homological coherence.
	
	Groups of cohomological dimension~$n$ over~$R$ are homologically $(n,\infty)$-coherent over~$R$~\cite[Proposition~VIII.6.1]{Brown82}.
	Over the ring~$\IZ$ (and suitably also over fields), this can be improved if the top-dimensional $\ell^2$-Betti number vanishes.

\begin{thm}[Fisher, Jaikin-Zapirain--Linton]
\label{thm:Fisher}
	Let~$n\in \IN$ with~$n\ge 1$ and
	let~$G$ be a locally indicable group with~$\cd_\IZ(G)=n$.
	If~$b^{(2)}_n(G)=0$,
	then~$G$ is homologically $(n-1,\infty)$-coherent over~$\IZ$.
	\begin{proof}
		The proof is analogous to that of~\cite[Proposition~2.3]{Fisher24}, see also~\cite[Theorem~3.10]{JZL}.
	\end{proof}
\end{thm}

\section{Higher coherence of some right-angled Artin groups}

We apply the inheritance results above to deduce the higher coherence of certain right-angled Artin groups.

\begin{proof}[Proof of Theorem~\ref{thm:RAAG}]
	We show that $(n,\infty)$-coherence is satisfied by the right-angled Artin groups associated to the flag simplicial complexes in Definition~\ref{defn:class}~\ref{item:simplices}, \ref{item:dimension}, \ref{item:Betti}, and~\ref{item:limit} and is preserved under the operations in Definition~\ref{defn:class}~\ref{item:gluing} and~\ref{item:cone}.
	
	\ref{item:simplices} If~$L$ is a simplex, then the group~$A_L$ is finitely generated free abelian and hence $(0,\infty)$-coherent and, in particular, $(n,\infty)$-coherent.
	
	\ref{item:dimension} If~$L$ is of dimension~$\le n-1$, then the group~$A_L$ is of geometric dimension~$\le n$ and hence $(n,\infty)$-coherent~\cite[Proposition~7.2.13]{Geoghegan08}.
	
	\ref{item:Betti} If~$L$ is of dimension~$n$ with~$H_n(L;\IQ)=0$, then the group~$A_L$ satisfies $\cd_\IZ(A_L)=n+1$ and $b^{(2)}_{n+1}(A_L)= \dim_\IQ H_n(L;\IQ)=0$ by~\cite{Davis-Leary03}.
	By Theorem~\ref{thm:Fisher}, the group~$A_L$ is homologically $(n,\infty)$-coherent over~$\IZ$  and, since~$n\ge 2$, hence $(n,\infty)$-coherent.
	
	\ref{item:limit} If~$L$ is a join~$L_1\ast \cdots\ast L_k$, where~$k\le n$ and~$L_1,\ldots,L_k$ are disjoint unions of simplices, then the group~$A_L$ splits as a direct product~$A_{L_1}\times \cdots\times A_{L_k}$.
	The groups~$A_{L_1},\ldots,A_{L_k}$ are free products of finitely generated free abelian groups and, in particular, they are limit groups.
	Then the group~$A_L$ is $(n,\infty)$-coherent by~\cite[Theorem~A]{BHMS09}, using that limit groups are of type~$\sfF_\infty$.
	
	\ref{item:gluing} For a gluing~$L\coloneqq L_1\cup_{L_0} L_2$ with~$L_1,L_2\in \calT_n$ and~$L_0\in \calT_1$ or~$L_0=\emptyset$, the group~$A_L$ splits as an amalgamated product~$A_{L_1}\ast_{A_{L_0}} A_{L_2}$.
	(Here~$A_{\emptyset}$ is the trivial group.)
	The group~$A_{L_0}$ is $(1,\infty)$-coherent by Theorem~\ref{thm:Droms}.
	By induction on the number of vertices, we may assume that the groups~$A_{L_1}$ and~$A_{L_2}$ are $(n,\infty)$-coherent.
	Then the group~$A_L$ is $(n,\infty)$-coherent by Theorem~\ref{thm:gog_coherence}.
	
	\ref{item:cone} For a cone~$CL\coloneqq L\ast \Delta^0$ with~$L\in \calT_n$, the group~$A_{CL}$ splits as a direct product~$A_L\times \IZ$.
	By induction on the number of vertices, we may assume that the group~$A_L$ is $(n,\infty)$-coherent.
	Then the group~$A_{CL}$ is $(n,\infty)$-coherent by Lemma~\ref{lem:extension_coherence}~\ref{item:extension_coherence}.
\end{proof}

We note that if a right-angled Artin group~$A_L$ splits as an amalgamated product over a non-trivial $(1,2)$-coherent group, then there exists a separating full subcomplex of~$L$ that lies in~$\calT_1$~\cite{Hull21}.

\begin{rem}
	Right-angled Artin groups are so-called graph products of~$\IZ$.
	Theorem~\ref{thm:RAAG} holds more generally for graph products of $(0,\infty)$-coherent groups when Definition~\ref{defn:class} is modified suitably.
	The $(1,2)$-coherence of graph products was studied in~\cite{Varghese19}.
	The $\ell^2$-Betti numbers and agrarian Betti numbers of graph products were computed in~\cite{DO12,FHL24}.
\end{rem}

Right-angled Artin groups can have subgroups with exotic finiteness properties, which are obvious obstructions to higher coherence.
Let~$L$ be a finite flag simplicial complex.
The \emph{Bestvina--Brady group~$H_L$} is defined as the kernel of the group homomorphism~$A_L\to \IZ$ that sends every standard generator of~$A_L$ to~$1\in \IZ$.
The finiteness properties of the group~$H_L$ are determined by the higher connectivity of~$L$ as follows.

\begin{thm}[Bestvina--Brady~\cite{Bestvina-Brady97}]
	Let~$L$ be a finite flag simplicial complex and let~$n\in \IN$ with~$n\ge 1$.
	Then the group~$H_L$ is of type~$\sfF_n$ if and only if~$L$ is $(n-1)$-connected.
\end{thm}

\begin{cor}
\label{cor:BB}
	Let~$L$ be a finite flag simplicial complex and let~$n\in \IN$ with~$n\ge 1$.
	If the right-angled Artin group~$A_L$ is $(n,n+1)$-coherent, then every full subcomplex of~$L$ that is $(n-1)$-connected is $n$-connected.
\end{cor}

For~$n=1$, the converse of Corollary~\ref{cor:BB} holds by Theorem~\ref{thm:Droms}.
For~$n\ge 2$, we do not know if the converse of Corollary~\ref{cor:BB} holds.

The $(3,\infty)$-coherence of the following example was pointed out to us by the referee.

\begin{ex}
	Let~$L$ be the flag simplicial complex that is the 2-fold suspension of~$\Delta^2\sqcup \Delta^0$.
	Then~$L$ is homotopy equivalent to a 2-dimensional sphere.
	The right-angled Artin group~$A_L$ is isomorphic to $F_2\times F_2\times (\IZ^3\ast \IZ)$.
	The group~$A_L$ is not $(1,2)$-coherent by Theorem~\ref{thm:Droms} because~$L$ contains a cycle of length~$4$ as a full subcomplex.
	The group~$A_L$ is also not $(2,3)$-coherent by Corollary~\ref{cor:BB} because~$L$ is $1$-connected and not $2$-connected.
	On the other hand, the group~$A_L$ is $(3,\infty)$-coherent by Theorem~\ref{thm:RAAG} because~$L$ is a join $(\Delta^0\sqcup \Delta^0)\ast (\Delta^0\sqcup \Delta^0)\ast (\Delta^2\sqcup \Delta^0)$ and hence lies in the class~$\calT_3$.
	In conclusion, the group~$A_L$ is $(n,n+1)$-coherent if and only if~$n\ge 3$.
	The $(4,\infty)$-coherence of the group~$A_L$ can also be deduced from Theorem~\ref{thm:RAAG} using that~$L$ is $4$-dimensional with~$H_4(L;\IQ)=0$ and hence lies in the class~$\calT_4$.
	
	For subdivisions of~$L$, the arguments for non-$(2,3)$-coherence and for $(4,\infty)$-coherence still apply but the $(3,\infty)$-coherence is unclear.
\end{ex}

\begin{ack}
	We thank Sam Fisher and Marco Linton for very helpful comments.
	We thank Sobhi Massalha for explanations about limit groups.
	We are extremely grateful to the referee for pointing us to the reference~\cite{BHMS09} and answering a question in a previous version.
	The first author was partially supported by the SFB~1085 \emph{Higher Invariants} (Universit\"at Regensburg, funded by the DFG).
\end{ack} 

\bibliographystyle{alpha}
\bibliography{bib}

\setlength{\parindent}{0cm}

\end{document}